\documentclass[preprint,review,10pt]{elsarticle}
\usepackage{amssymb}
\usepackage{amsfonts}
\usepackage{amsmath}
\usepackage{amsthm}
\usepackage{color}
\usepackage{graphicx}
\usepackage{epsfig,mathrsfs}
\usepackage[bf,SL,BF]{subfigure}
\usepackage{fancyhdr}
\usepackage{CJK}
\usepackage{caption}
\usepackage{wrapfig}
\usepackage{cases}
\usepackage{subfigure}
\graphicspath{{figure/}}
\usepackage{setspace}

\newtheorem{theorem}{Theorem}[section]
\newtheorem{lemma}{Lemma}[section]
\newtheorem*{lemma*}{Lemma A}

\numberwithin{equation}{section}

\renewcommand{\eqref}[1]{(\ref{#1})}

 \allowdisplaybreaks

\begin{document}
%\selectlanguage{english}
 \pagenumbering{arabic}
\begin{spacing}{0.85}
\begin{frontmatter}
\title{{\bf \noindent A second-order accurate numerical method for the  space-time tempered fractional  diffusion-wave equation}}
\author{Minghua Chen$^{*}$, Weihua Deng}
\cortext[cor2]{Corresponding author. E-mail: chenmh@lzu.edu.cn;  dengwh@lzu.edu.cn}
%\author{Weihua Deng\cortext[cor2]{Corresponding author.}\ead{dengwh@lzu.edu.cn}
%%{\cortext[cor2]{Corresponding author.}\ead{dengwh@lzu.edu.cn}\author{Weihua Deng\corref{cor2}}}
%%\author{\corref{}}
%\author{Yujiang Wu}
\address{School of Mathematics and Statistics, Gansu Key Laboratory of Applied Mathematics and Complex Systems, Lanzhou University, Lanzhou 730000, P.R. China}

%\date{}

\begin{abstract}
 %We  study the  numerical method for the Volterra type  integro-differential
%equations which can be written as  the space-time tempered fractional  diffusion-wave equation. 
This paper focuses on providing the high order algorithms for the space-time tempered fractional  diffusion-wave equation.
The designed schemes are unconditionally stable and have the global truncation error $\mathcal{O}(\tau^2+h^2)$, being theoretically proved and numerically   verified.

\medskip
\noindent {\bf Keywords:}
Space-time tempered fractional  diffusion-wave equation;
Integro-differential equation;
Numerical stability and convergence
\end{abstract}
\end{frontmatter}

\section{Introduction}
We study a second-order accurate numerical method in both space and time  for the integro-differential equation whose prototype is, for $1< \alpha,\gamma\leq 2, \lambda\geq 0$,
%\begin{equation}\label{1.1}
%\frac{\partial}{\partial t}u(x,t)=  \frac{1}{\Gamma(\gamma-1)}\int_{0}^t{\left(t-\tau\right)^{\gamma-2}}{{\nabla}_x^{\alpha} u(x,\tau)}d\tau,~~1<\alpha, \gamma<2,
%\end{equation}
\begin{equation} \label{1.1}
\frac{\partial}{\partial t}u(x,t)= I_t^{\gamma-1,\lambda} {\nabla}_x^{\alpha}u(x,t)
= \frac{1}{\Gamma(\gamma-1)}\int_{0}^t{\left(t-\tau\right)^{\gamma-2}}e^{-\lambda(t-\tau)}{{\nabla}_x^{\alpha} u(x,\tau)}d\tau,\\
\end{equation}
with the initial condition  $u(x,0)=u_0(x)$, $x \in \Omega=(a,b)$ and the homogeneous Dirichlet boundary conditions, characterizing the propagation of wave with the tempered power law decay.
Here  the tempered fractional  integral $I_t^{\beta,\lambda} $ with $\beta=\gamma-1>0$ is defined as \cite{Cartea:07,Chen:15}
\begin{equation}\label{1.3}
I_t^{\beta,\lambda} u(x,t)=\frac{1}{\Gamma(\beta)}\int_{0}^t{\left(t-\tau\right)^{\beta-1}}e^{-\lambda(t-\tau)}{u(x,\tau)}d\tau, ~~~~t>0.
\end{equation}
The Riesz fractional derivative with $\alpha \in (1,2)$,
is defined as  \cite{Podlubny:99}
\begin{equation}\label{1.2}
{\nabla}_x^{\alpha}u(x,t) =-\kappa_{\alpha}\left( _{a}D_x^{\alpha}+ _{x}\!D_{b}^{\alpha} \right)u(x,t)  ~~{\rm with}~~\kappa_{\alpha}=\frac{1}{2\cos(\alpha \pi/2)},
\end{equation}
\begin{equation*}
\begin{split}
 _{a}D_x^{\alpha}u(x,t)
 & =
\frac{1}{\Gamma(2-\alpha)} \displaystyle \frac{\partial^2}{\partial x^2}
 \int_{a}\nolimits^x{\left(x-\xi\right)^{1-\alpha}}{u(\xi,t)}d\xi,
 ~~
 _{x}D_{b}^{\alpha}u(x,t)=
 \frac{1}{\Gamma(2-\alpha)}\frac{\partial^2}{\partial x^2}
\int_{x}\nolimits^{b}{\left(\xi-x\right)^{1-\alpha}}{u(\xi,t)}d\xi.
\end{split}
\end{equation*}

It can be noted that, if $\lambda=0$,  (\ref{1.1}) reduces to the following  space-time fractional  diffusion-wave equation \cite{Fujita:1990},
\begin{equation*}
\begin{split}
{^c}\!{D}_t^\gamma u(x,t)= {\nabla}_x^{\alpha} u(x,t)~~{\rm for}~~ 1<\alpha, \gamma\leq 2.
\end{split}
\end{equation*}
%and space fractional  diffusion-wave equation
%\begin{equation*}
%\begin{split}
%\frac{\partial^2}{\partial t^2}u(x,t)= {\nabla}_x^{\alpha} u(x,t)~~{\rm for}~~ 1<\alpha<2, ~~, \gamma=2.
%\end{split}
%\end{equation*}

Numerical methods for the time discretization of (\ref{1.1}) with $\lambda=0, \alpha=2$, i.e., the  time fractional diffusion-wave equation,  have been proposed  by various authors \cite{Cuesta:06,Liu:13,Mustapha:13,Sun:06,Yang:14,Zeng:15}. For the time fractional diffusion-wave equation,
there are already several types of second-order discretization   schemes.   For example,  based on the second-order fractional Lubich's methods \cite{Lubich:86},
Cuesta (2006) et al derive the second-order error bounds of the time discretization in a Banach space with ${\nabla}_x^{2}$ as a  sectorial operator  \cite{Cuesta:06}; 
and Yang (2014) et al obtain the  second-order convergence  schemes with $1\leq \gamma \leq 1.71832$ \cite{Yang:14}.
McLean and Mustapha (2007) study the Crank-Nicolson scheme for the time discretization  with the non-uniform grid in time direction \cite{Mclean:07}.
Using the fractional trapezoidal rule, Zeng (2015)  obtains  the second-order schemes \cite{Zeng:15}.

For the space-time tempered  fractional diffusion-wave equation  of (\ref{1.1}) with $\lambda=0$,
Mainardi (2001) et al obtain the fundamental solution of the space-time fractional diffusion equation \cite{Mainardi:2001};
Metzler and Nonnenmacher (2002) investigate the physical background and implications of a space-and time-fractional diffusion and wave equations \cite{{Metzler:02}}.
Recently, the numerical solutions of space-time fractional diffusion-wave equations
and space fractional diffusion-wave equations are, respectively,  discussed  in \cite{Garg:14} and \cite{Deng:15}.
However, it seems that achieving a second-order accurate scheme for (\ref{1.1}) is not an easy task.
This paper focuses on providing effective and highly accurate numerical algorithms for the  space-time tempered fractional  diffusion-wave equation   (\ref{1.1}).
The designed schemes are unconditionally stable and have the global truncation error $\mathcal{O}(\tau^2+h^2)$, being theoretically proved and numerically   verified.
It  can be easily extended to the problems discussed in \cite{Deng:15,Garg:14,Yang:14}.

The rest of the paper is organized as follows. In the next section, we propose the second-order algorithm to the model. In Sec. 3, we do the detailedly theoretical analyses for the stability and convergence with the second order accuracy in both time and space directions for the derived schemes. To verify the theoretical results, especially the convergence orders, the extensive numerical experiments are performed in Sec. 4. The paper is concluded with some remarks in the last section.

\section{High order schemes for the  space-time tempered fractional  diffusion-wave equation}
Let the mesh points $x_i=ih$ for $i=0,1,\ldots,M$, and $t_n=n\tau$ for 
$n=0,1,\ldots,N$,   where $h=b/M$ and $\tau=T/N$ are the uniform space stepsize and time steplength, respectively.
Denote $u_i^n$ as the numerical approximation to $u(x_i,t_n)$.
Nowadays, there are already several types of high
order discretization schemes for the Riemann-Liouville space fractional derivatives \cite{Chen:0013,Chen:14,Hao:15,Ortigueira:06,Sousa:12,Tian:12}.
Here, we utilize the second-order formula \cite{Chen:14}  to approximate  the Riesz fractional derivative (\ref{1.2}),  that is
\begin{equation}\label{2.1}
\begin{split}
 \nabla_x^{\alpha} u(x,t)|_{x=x_i}
&=-\frac{\kappa_{\alpha}}{\Gamma(4-\alpha)h^\alpha} \sum_{j=1}^{M-1}w_{i,j}^{\alpha}u(x_{j},t)   +\mathcal{O}(h^2)
\end{split}
\end{equation}
with  $i=1,\ldots,M-1$, where
\begin{equation*}
w_{i,j}^{\alpha}=\left\{ \begin{array}
 {l@{\quad } l}
  w_{i-j+1}^{\alpha},&j < i-1,\\
  w_{0}^{\alpha}+w_{2}^{\alpha} ,&j=i-1,\\
 2w_{1}^{\alpha},&j=i,\\
w_{0}^{\alpha}+w_{2}^{\alpha} ,&j=i+1,\\
w_{j-i+1}^{\alpha} ,&j>i+1,
 \end{array}
 \right.~~~~{\rm and }~~~~
w_m^{\alpha}=\left\{ \begin{array}
 {l@{\quad } l}

1,&m =0,\\

-4+2^{3-\alpha},&m =1,\\

6-2^{5-\alpha}+3^{3-\alpha},&m =2,\\

(m+1)^{3-\alpha}-4m^{3-\alpha}+6(m-1)^{3-\alpha}\\
~~~~~~~~~~~-4(m-2)^{3-\alpha}+(m-3)^{3-\alpha},&m \geq 3.\\
 \end{array}
 \right.
\end{equation*}
Further denoting   $u^n=\left[u_1^n,u_2^n,\ldots,u_{M-1}^n\right]^T$,
from (\ref{2.1}), then we obtain
\begin{equation}\label{2.2}
\begin{split}
 \nabla_h^{\alpha} u_i^n
=-\frac{\kappa_{\alpha}}{{\Gamma(4-\alpha)}h^\alpha} \sum_{j=1}^{M-1}w_{i,j}^{\alpha}u_j^n
~~{\rm and}~~
 \nabla_h^{\alpha} u^n
=-\frac{\kappa_{\alpha}}{{\Gamma(4-\alpha)}h^\alpha} A_\alpha u^n,
\end{split}
\end{equation}
where the matrix
\begin{equation}\label{2.3}
A_\alpha=\left [ \begin{matrix}
2w_{1}^{\alpha} &w_{0}^{\alpha}+w_{2}^{\alpha}&w_{3}^{\alpha}    &      \cdots   & w_{M-2}^{\alpha}    &   w_{M-1}^{\alpha}   \\
w_{0}^{\alpha}+w_{2}^{\alpha}&  2w_{1}^{\alpha}  &w_{0}^{\alpha}+w_{2}^{\alpha}&  w_{3}^{\alpha}   &     \cdots   & w_{M-2}^{\alpha}\\
w_{3}^{\alpha}            &w_{0}^{\alpha}+w_{2}^{\alpha}&2w_{1}^{\alpha}       & w_{0}^{\alpha}+w_{2}^{\alpha}&     \ddots              & \vdots  \\
\vdots                   &          \ddots         &       \ddots            &        \ddots            &      \ddots             &  w_{3}^{\alpha} \\
  w_{M-2}^{\alpha} &   \ddots         &       \ddots        &        \ddots            &   2w_{1}^{\alpha}       & w_{0}^{\alpha}+w_{2}^{\alpha} \\
 w_{M-1}^{\alpha}    &   w_{M-2}^{\alpha}  &   \cdots          &         \cdots           &w_{0}^{\alpha}+w_{2}^{\alpha}& 2w_{1}^{\alpha}  
 \end{matrix}
 \right ].
\end{equation}
We know that the tempered fractional integral (\ref{1.3})  has the second-order approximation \cite{Chen:15}
\begin{equation}\label{2.4}
I_t^{\beta,\lambda}  u(x,t)|_{t=t_n}=
\frac{1}{\Gamma(\beta)}\int_{0}^{t_n}{\left(t_n-\tau\right)^{\beta-1}}e^{-\lambda(t_n-\tau)} {u(x,\tau)}d\tau
=\tau^{\beta}\sum_{k=0}^{n}l_k^{\beta}u(x,t_{n-k})+\mathcal{O}(\tau^2),
\end{equation}
where $l_k^{\beta}$ are the coefficients of the Taylor expansions of the generating function
\begin{equation}\label{2.5}
l^{\beta}(z)= \left(1-\frac{z}{e^{\lambda \tau}}\right) ^{-\beta}\left(  1+  \frac{1}{2}\left( 1- \frac{z}{e^{\lambda \tau}} \right)\right) ^{-\beta}= \sum_{k=0}^{\infty}l_k^{\beta}z^k
\end{equation}
with
\begin{equation}\label{2.55}
 l_k^{\beta} =e^{-\lambda k \tau}\left(\frac{3}{2}\right)^{-\beta} \sum_{m=0}^{k} 3^{-m}g_m^{-\beta}g_{k-m}^{-\beta},~~\beta=\gamma-1\in(0,1].
\end{equation}
Without loss of generality, we suppose (\ref{1.1}) with the  zero initial value \cite{Ji:15} and  add a force term $f(x,t)$ on the right side of  (\ref{1.1}).
Considering  (\ref{1.1}) at the point $(x_i,t_{n+\frac{1}{2}})$,  there exists
\begin{equation} \label{2.6}
\frac{\partial}{\partial t}u(x_i,t_{n+\frac{1}{2}})
=  \frac{1}{\Gamma(\gamma-1)}\int_{0}^{t_{n+\frac{1}{2}}}{\left(t_{n+\frac{1}{2}}-\tau\right)^{\gamma-2}}{{\nabla}_x^{\alpha} u(x_i,\tau)}d\tau+f\left(x_i,t_{n+\frac{1}{2}}\right).
\end{equation}
According to  (\ref{2.1}) and (\ref{2.4}) and taking  $\beta=\gamma-1$, we can write (\ref{2.6}) as
\begin{equation}\label{2.7}
\begin{split}
\frac{u(x_i,t_{n+1})-u(x_i,t_{n})}{\tau}
=\frac{\tau^{\beta}}{2}\sum_{k=0}^{n}l_k^{\beta} \nabla_x^{\alpha} \left( u(x_i,t_{n+1-k})+u(x_i,t_{n-k})\right)
+f\left(x_i,t_{n+\frac{1}{2}}\right)+\mathcal{O}(\tau^2+h^2).
\end{split}
\end{equation}
Multiplying (\ref{2.7}) by $\tau$, we have the following equation
\begin{equation}\label{2.8}
\begin{split}
u(x_i,t_{n+1})-u(x_i,t_{n})=\frac{\tau^{\gamma}}{2}\sum_{k=0}^{n}l_k^{\beta} \nabla_x^{\alpha} \left( u(x_i,t_{n+1-k})+u(x_i,t_{n-k})\right)
+\tau f\left(x_i,t_{n+\frac{1}{2}}\right)+ R_i^{n+1}
\end{split}
\end{equation}
with the residual term
\begin{equation}\label{2.9}
  |R_i^n| \leq C_u\tau(\tau^2+h^2),
\end{equation}
and  $C_u$ is a positive constant independent of $\tau$ and $h$. 
Then the full discretization of (\ref{2.8}) has the following form
\begin{equation}\label{2.10}
\begin{split}
u_i^{n+1}-u_i^n
=&\frac{\tau^{\gamma}}{2}\sum_{k=0}^{n}l_k^{\beta}  \nabla_h^{\alpha}\left( u_i^{n+1-k} + u_i^{n-k}\right)+\tau f_i^{n+\frac{1}{2}}.
\end{split}
\end{equation}

\section{Stability and convergence}
In this section, we prove that the scheme (\ref{2.10}) is unconditionally stable and convergent in discrete $L^2$ norm. Denote the grid functions  $u^n=\{u_i^n| 0 \leq i \leq M, n \geq 0 \}$
and $v^n=\{v_i^n| 0 \leq i \leq M, n \geq 0 \}$; and
\begin{equation*}
\begin{split}
&(u^n,v^n)=h\sum_{i=1}^{M-1}u_i^nv_i^n, ~~~~~||u^n||=(u^n,u^n)^{1/2}.
\end{split}
\end{equation*}
\begin{lemma} [\cite{Chen:14,Wang:15}] \label{lemma3.1}
 Let  $A_\alpha $ be given in (\ref{2.3}) with $1<\alpha<2$. Then  there exists an operator $\Lambda^\alpha$ satisfies
\begin{equation*}
-  (A_\alpha u,u)>0~~{\rm and}~~-  ( A_\alpha  u,v)=( \Lambda^\alpha u,\Lambda^\alpha v).
\end{equation*}
\end{lemma}

\begin{lemma}\label{lemma3.3}
Let $l_k^{\beta}$ be defined by (\ref{2.55}) with $\beta=\gamma-1$.
Then $l_k^{\beta}\geq 0$, $\forall k\geq0.$
\end{lemma}
\begin{proof}
  According to  (\ref{2.5}) and \cite{Yang:14}, the desired results are obtained.
\end{proof}

\begin{lemma} \label{lemma3.4}
Let $l_k^{\beta}$ be defined by (\ref{2.55}) with $\beta=\gamma-1$, $1<\gamma\leq 2$.
Then for any positive integer $N$ and  real vector $\left( v_i^0,v_i^1,\ldots,v_i^N \right) \in \mathbb{R}^{N+1}$, it holds that
$$\sum_{n=0}^{N}\left(\sum_{k=0}^{n}l_k^{\beta}  v_i^{n-k}\right)v_i^n \geq 0,~~i=1,2,\ldots, M-1.$$
\end{lemma}
\begin{proof}
By the mathematical induction method, we can prove that
\begin{equation*}
\sum_{n=0}^{N}\left(\sum_{k=0}^{n}l_k^{\beta}  v_i^{n-k}\right)v_i^n=V_iL^{\beta}V_i^T, ~~i=1,2,\ldots, M-1,
\end{equation*}
 where
 $$V_i=\left( v_i^0,  v_i^1,\ldots,  v_i^{N-1}, v_i^N \right)$$
  and the real symmetric matrix
 \begin{equation}\label{3.1}
L^{\beta}=\left [ \begin{matrix}
l^{\beta}_0  &\frac{l^{\beta}_1}{2}&\frac{l^{\beta}_2}{2}    &      \cdots   & \frac{l^{\beta}_{N-1}}{2}     &  \frac{l^{\beta}_N}{2}  \\
\frac{l^{\beta}_1}{2}&  l^{\beta}_0   &\frac{l^{\beta}_1}{2}&  \frac{l^{\beta}_2}{2}   &     \cdots   & \frac{l^{\beta}_{N-1}}{2} \\
\frac{l^{\beta}_2}{2}            &\frac{l^{\beta}_1}{2}&l^{\beta}_0        & \frac{l^{\beta}_1}{2}&     \ddots              & \vdots  \\
\vdots                   &          \ddots         &       \ddots            &        \ddots            &      \ddots             &  \frac{l^{\beta}_2}{2} \\
  \frac{l^{\beta}_{N-1}}{2}  &   \ddots         &       \ddots        &        \ddots            &   l^{\beta}_0        & \frac{l^{\beta}_1}{2} \\
\frac{l^{\beta}_N}{2}   &   \frac{l^{\beta}_{N-1}}{2}   &   \cdots          &         \cdots           &\frac{l^{\beta}_1}{2}& l^{\beta}_0
 \end{matrix}
 \right ].
\end{equation}
Next we prove that the real symmetric  matrix  $L^{\beta}$ defined in (\ref{3.1}) is positive semi-definite.
With $J=\sqrt{-1}$, we know that  the generating function \cite[p. 12-14]{Chan:07} of $L^{\beta}$ is 
\begin{equation}\label{3.2}
\begin{split}
 f(\beta,x)
  &=\frac{1}{2}\sum_{k=0}^{\infty}l^{\beta}_ke^{Jkx}+\frac{1}{2}\sum_{k=0}^{\infty}l^{\beta}_ke^{-Jkx}=\sum_{k=0}^{\infty}l^\beta_k\cos(kx)=\frac{1}{2}l^{\beta}(e^{Jx})+\frac{1}{2}l^{\beta}(e^{-Jx})\\
  &=\frac{1}{2}\left(1-\frac{e^{Jx}}{e^{\lambda \tau}}\right) ^{-\beta}\left(  1+  \frac{1}{2}\left( 1- \frac{e^{Jx}}{e^{\lambda \tau}} \right)\right) ^{-\beta}
  +\frac{1}{2}\left(1-\frac{e^{-Jx}}{e^{\lambda \tau}}\right) ^{-\beta}\left(  1+  \frac{1}{2}\left( 1- \frac{e^{-Jx}}{e^{\lambda \tau}} \right)\right) ^{-\beta}.
\end{split}
\end{equation}
Since $ f(\beta,x)$  is an even function and $2\pi$-periodic continuous real-valued functions defined on $[-\pi,\pi]$, we just need to consider its principal value on $[0,\pi]$.
Next we prove that $f(\beta,x)$ defined in (\ref{3.2}) is nonnegative. Denoting $d=e^{\lambda \tau}\geq 1$, we have
\begin{equation*}
\begin{split}
\left(1-\frac{e^{\pm Jx}}{e^{\lambda \tau}}\right)^{-\beta}
&=d^{\beta}\left(d-e^{\pm Jx}\right)^{-\beta}
=d^{\beta}\left((d-\cos x)^2+\sin^2 x  \right)^{-\frac{{\beta}}{2}} e^{\pm J\beta\theta_1};\\
\left(1 + \frac{1}{2}(1-e^{\pm Jx})\right)^{-\beta}
&=(2d)^{\beta}\left(3d-e^{\pm Jx}\right)^{-\beta}=(2d)^{\beta}\left((3d-\cos x)^2+\sin^2 x  \right)^{-\frac{{\beta}}{2}} e^{\pm J\beta\theta_2};
  \end{split}
\end{equation*}
where
$$\theta_1=\arctan\left( \frac{\sin x}{d-\cos x} \right)~~{\rm and}~~\theta_2=\arctan\left( \frac{\sin x}{3d-\cos x} \right).$$
It yields
\begin{equation*}
\begin{split}
 f(\beta,x)=(2d^2)^{\beta}\left((d-\cos x)^2+\sin^2 x  \right)^{-\frac{{\beta}}{2}}\left((3d-\cos x)^2+\sin^2 x  \right)^{-\frac{{\beta}}{2}}
  \cos  \beta\left( \theta_1+\theta_2  \right ).
\end{split}
\end{equation*}

When $x=0$, according to Lemma  \ref{lemma3.3} and Eq. (\ref{3.2}), we have  $f(\beta,x)=\sum_{k=0}^{\infty}l_k^{\beta}\geq 0$.

When $x=\pi$, using  (\ref{3.2})  and (\ref{2.5}), we have
$$f(\beta,x)=\sum_{k=0}^{\infty}l^\beta_k\cos(k\pi)=\sum_{k=0}^{\infty}(-1)^kl^\beta_k
=\left(1+\frac{1}{e^{\lambda \tau}}\right) ^{-\beta}\left(  1+  \frac{1}{2}\left( 1+ \frac{1}{e^{\lambda \tau}} \right)\right) ^{-\beta}>0.$$
Next we consider $x\in(0,\pi)$.
Using  $\tan \left(\frac{x}{2}\right)=\frac{\sin x}{1+\cos x}>\frac{\sin x}{3-\cos x}:=\tan \theta_3\geq 0$ and
$$0\leq \tan(\theta_1)=\frac{\sin x}{d-\cos x}\leq \frac{\sin x}{1-\cos x}=\tan\left(\frac{\pi}{2}-\frac{x}{2}  \right),~~
0\leq \tan(\theta_2)=\frac{\sin x}{3d-\cos x}\leq \frac{\sin x}{3-\cos x}=\tan(\theta_3),  $$
 we get $ 0 \leq \theta_1+\theta_2\leq \frac{\pi}{2}-\frac{x}{2}+\theta_3\leq \frac{\pi}{2}$.
Hence $ f(\beta,x)\geq 0$ for $\beta\in[-1,1].$

From the Grenander-Szeg\"{o} theorem \cite[p. 13-14]{Chan:07},
it implies that $L^{\beta}$  is a real symmetric positive semi-definite matrix. The proof is completed.
\end{proof}

\begin{theorem}\label{theorem3.9}
The difference scheme (\ref{2.10})  with $1<\alpha, \gamma \leq 2$ is unconditionally stable.
\end{theorem}
\begin{proof}
Let $\widetilde{u}_{i}^n~(i=0,1,\ldots,M;\,n=0,1,\ldots,N)$ be the approximate solution of $u_i^n$,
which is the exact solution of the difference scheme (\ref{2.10}).
Putting $\epsilon_i^n=\widetilde{u}_i^n-u_i^n$, and denoting $\epsilon^n=[\epsilon_0^n,\epsilon_1^n,\ldots, \epsilon_{M}^n]$,
 then from (\ref{2.10}) we obtain the following perturbation equation
 \begin{equation}\label{3.3}
\begin{split}
\epsilon_i^{n+1}-\epsilon_i^n
=&\frac{\tau^{\gamma}}{2}\sum_{k=0}^{n}l_k^{\beta}  \nabla_h^{\alpha}\left( \epsilon_i^{n+1-k} + \epsilon_i^{n-k}\right).
\end{split}
\end{equation}
Multiplying (\ref{3.3}) by $h\left( \epsilon_i^{n+1}+\epsilon_i^n \right)$ and summing up for $i$ from $1$ to $M-1$, we have
\begin{equation*}
||\epsilon^{n+1}||^2-||\epsilon^n||^2=\frac{\tau^{\gamma}}{2}\sum_{k=0}^{n}l_k^{\beta}\left( \nabla_h^{\alpha} \left(e^{n+1-k}+ e^{n-k}\right),e^{n+1}+e^n\right).
\end{equation*}
Summing up for $n$ from $0$ to $N$  on both sides of the above equation, it yields
  \begin{equation}\label{3.4}
||e^{N+1}||^2-||e^0||^2=\frac{\tau^{\gamma}}{2} \sum_{n=0}^{N}\sum_{k=0}^{n}l_k^{\beta}\left( \nabla_h^{\alpha} \left(e^{n+1-k}+ e^{n-k}\right),e^{n+1}+e^n\right).
\end{equation}
According to (\ref{2.2}) and Lemmas \ref{lemma3.1}, \ref{lemma3.4}, we get
\begin{equation}\label{3.5}
\begin{split}
&\frac{\tau^{\gamma}}{2} \sum_{n=0}^{N}\sum_{k=0}^{n}l_k^{\beta}\left( \nabla_h^{\alpha} \left(e^{n+1-k}+ e^{n-k}\right),e^{n+1}+e^n\right)\\
&=\frac{|\kappa_{\alpha}|\tau^{\gamma}}{2{h^\alpha}}\sum_{n=0}^{N}\sum_{k=0}^{n}l_k^{\beta}\left(  A_\alpha \left(e^{n+1-k}+ e^{n-k}\right),e^{n+1}+e^n\right)\\
&=-\frac{|\kappa_{\alpha}|\tau^{\gamma}}{2{h^\alpha}}\sum_{n=0}^{n}\sum_{k=0}^{n}l_k^{\beta}\left(  \Lambda^\alpha \left(e^{n+1-k}+ e^{n-k}\right), \Lambda^\alpha (e^{n+1}+e^n)\right)
\leq 0.
\end{split}
\end{equation}
Using  (\ref{3.4}) and (\ref{3.5}), for any positive integer $N$, it yields $||e^{N}||\leq ||e^0||.$
The proof is completed.
\end{proof}

\begin{theorem}\label{theorem4.4}
Let $u(x_i,t_n)$ be the exact solution of (\ref{1.1}) with $1<\alpha, \gamma \leq 2$, and $u_i^n$ the  solution of
the finite difference scheme (\ref{2.10}). Then
\begin{equation*}
  \begin{split}
||u(x_i,t_n)-u_i^n||_2 \leq   2 C_ub^\frac{1}{2}T(\tau^2+h^2), \quad i=0,1,\ldots,M;\,n=0,1,\ldots,N,
  \end{split}
  \end{equation*}
where $C_u$ is defined by  (\ref{2.9}) and  $(x_i,t_n)\in (0,b) \times (0,T]$ with $N\tau\leq T$.
\end{theorem}

\begin{proof}
Denote $e_i^n=u(x_i,t_n)-u_i^n$ and    $e^n=[e_0^n, e_1^n,\ldots, e_{M}^n]^T$.
 Subtracting (\ref{2.10}) from (\ref{2.8}) and using $e^0=0$, we obtain
\begin{equation}\label{3.6}
e_i^{n+1}-e_i^n
=\frac{\tau^{\gamma}}{2}\sum_{k=0}^{n}l_k^{\beta} \nabla_h^{\alpha}\left( e_i^{n+1-k}+ e_i^{n-k}\right)+R_i^{n+1}.
\end{equation}
Multiplying (\ref{3.6}) by $h\left( e_i^{n+1}+e_i^n \right)$ and summing up for $i$ from $1$ to $M-1$, we have
\begin{equation*}
||e^{n+1}||^2-||e^n||^2=\frac{\tau^{\gamma}}{2}\sum_{k=0}^{n}l_k^{\beta}\left( \nabla_h^{\alpha} \left(e^{n+1-k}+ e^{n-k}\right),e^{n+1}+e^n\right)+\left(R^{n+1},e^{n+1}+e^n \right).
\end{equation*}
Replacing $n$ with $s$, there exists
  \begin{equation*}
||e^{s+1}||^2-||e^s||^2=\frac{\tau^{\gamma}}{2}\sum_{j=0}^{s}l_j^{\beta}\left( \nabla_h^{\alpha} \left(e^{s+1-j}+ e^{s-j}\right),e^{s+1}+e^s\right)+\left(R^{s+1},e^{s+1}+e^s \right).
\end{equation*}
Summing up for $s$ from $0$ to $n$ and using (\ref{3.5}), there exists
\begin{equation*}
\begin{split}
||e^{n+1}||^2
&=\frac{\tau^{\gamma}}{2}\sum_{s=0}^{n}\sum_{j=0}^{s}l_j^{\beta}\left( \nabla_h^{\alpha} \left(e^{s+1-j}+ e^{s-j}\right),e^{s+1}+e^s\right)+\sum_{s=0}^{n}\left(R^{s+1},e^{s+1}+e^s \right)\\
&\leq \sum_{s=0}^{n}\left(R^{s+1},e^{s+1}+e^s \right).
\end{split}
\end{equation*}
Using (\ref{2.9}) and above inequality and the Cauchy-Schwarz inequality, it yields
\begin{equation*}\begin{split}
||e^{n+1}||^2
&\leq h\sum_{s=0}^{n}\sum_{i=1}^{M-1}|R_i^{s+1}|\cdot \left(|e_i^{s+1}|+|e_i^s| \right) \\
%&\leq C_u\tau(\tau^2+h^2)\sum_{s=0}^{n}\sum_{i=1}^{M-1}h \left(|e_i^{s+1}|+|e_i^s| \right)\\
&\leq C_uT(\tau^2+h^2)\sum_{i=1}^{M-1}h \left(|e_i^{s+1}|+|e_i^s| \right)\\
&\leq C_uT(\tau^2+h^2)\left(\sum_{i=1}^{M-1}\left|\sqrt{h}\right|^2 \right)^\frac{1}{2}
\left[\left(\sum_{i=1}^{M-1}\left|\sqrt{h}e_i^{s+1}\right|^2\right)^\frac{1}{2}
+\left(\sum_{i=1}^{M-1}\left|\sqrt{h}e_i^{s}\right|^2\right)^\frac{1}{2}\right] \\
&\leq C_ub^\frac{1}{2}T(\tau^2+h^2)\left(||e^{s+1}||+||e^s|| \right)\\
&\leq 2\sigma C_ub^\frac{1}{2}T(\tau^2+h^2),
\end{split}
\end{equation*}
where $\sigma=\max\limits_{0\leq s\leq n+1}||e^s||$.
Taking the maximum over $n$ on both sides of above equation, there exists $\sigma^2 \leq  2\sigma C_ub^\frac{1}{2}T(\tau^2+h^2)$,
which leads to $\sigma \leq  2 C_ub^\frac{1}{2}T(\tau^2+h^2)$.
Hence $$||e^n||\leq \max\limits_{0\leq s\leq n+1}||e^s||\leq  2 C_ub^\frac{1}{2}T(\tau^2+h^2).$$
The proof is completed.
\end{proof}

%%%%%%%%%%%%%%%%%%%%%%%%%%%%%%%%%%%%%%%%%%%%%%%%%%%%%%%%%%
\section{Numerical Results}
Consider the integro-differential equation  (\ref{1.1}) on a finite domain  $0< x < 1 $,  $0<t \leq 1/2$.
Without loss of generality, we add a force term $f(x,t)$ on the right side of  (\ref{1.1}).
Then  the forcing function is
\begin{equation*}
\begin{split}
 f(x,t)=&\left(3e^{-\lambda t}t^2-\lambda e^{-\lambda t}t^3\right)x^2(x-1)^2   +\frac{\Gamma(4)}{2\Gamma(3+\gamma)\cos(\alpha \pi/2)}e^{-\lambda t}t^{2+\gamma}\\
 & \times \left[ 2\frac{x^{2-\alpha}+(1-x)^{2-\alpha}}{\Gamma(3-\alpha)}   -12\frac{x^{3-\alpha}+(1-x)^{3-\alpha}}{\Gamma(4-\alpha)}
   +24\frac{x^{4-\alpha}+(1-x)^{4-\alpha}}{\Gamma(5-\alpha)} \right],
  \end{split}
\end{equation*}
the initial condition $u(x,0)=0 $, and the boundary conditions $u(0, t)=u(1,t)=0$. And (\ref{1.1}) has the exact
solution $$u(x,t)=e^{-\lambda t}t^3x^2(1-x)^2. $$
\begin{table}[h]\fontsize{9.5pt}{12pt}\selectfont%生成浮动表格
 \begin{center}%\def\tabcolsep{28.5pt}%表格居中
  \caption {The maximum errors and convergence orders for  (\ref{2.10}) with  $h=\tau$ and $\lambda=0.1$.} \vspace{5pt}% 标题，离表格一定的距离
\begin{tabular*}{\linewidth}{@{\extracolsep{\fill}}*{10}{c}}                                    \hline  %画顶端的横线
%%%%%%%%%%%%%%%%%%%%%%%%%%%%%%%%%%%%%%%%%%%%%%%%%%%%%%%%%%%%%%%%%%%%%%%%%%%%%%%%%%%%%%%%%%%%%%%%%%%%%%%%%%%%%%%%%%%%%%%%%%%%%%%%%%%%
$\tau$& $\gamma=2,\alpha=1.5$&  Rate       & $\gamma=1.3,\alpha=1.7$  & Rate       & $\gamma=1.7,\alpha=1.3$ &   Rate    \\\hline
~~~1/20&   5.2886e-05   &           & 3.8119e-05     &          & 4.7519e-05    &     \\
~~~1/40&   1.4084e-05   &  1.91     & 9.7938e-06     & 1.96     & 1.2539e-05    & 1.92 \\
~~~1/80&   3.6352e-06   &  1.95     & 2.4815e-06     & 1.98     & 3.2206e-06    & 1.96  \\
~~~1/160&  9.2322e-07   &  1.98     & 6.2446e-07     & 1.99     &  8.1607e-07   & 1.98   \\ \hline
    \end{tabular*}\label{table:1}%\vspace{-15pt}
  \end{center}
\end{table}

Table \ref{table:1}  shows  the maximum errors  at time $T=1/2$
with $h=\tau$; and  the numerical results confirm that the scheme (\ref{2.10}) has the global truncation error $\mathcal{O} (\tau^2+h^2)$.

\section{Conclusion}
With numerical experiments and detailed theoretical analysis, we construct the second-order schemes for the space-time tempered fractional diffusion-wave equation. The corresponding algorithms, theoretical and numerical results can also be extended to the problems discussed in \cite{Deng:15,Garg:14,Yang:14}.

\section*{Acknowledgments}This work was supported by NSFC 11601206 and 11671182, the Fundamental Research Funds for the Central Universities under Grant No. lzujbky-2016-105.

%\end{spacing}

{
\small
\bibliographystyle{ieee}
\bibliography{CASSreference}

%\begin{thebibliography}{99}

}
\end{spacing}
\end{document}